\documentclass[a4paper,UKenglish,cleveref, autoref, thm-restate]{lipics-v2021}

\usepackage{amssymb,amsmath,amsthm,times,mathrsfs}
\usepackage{tikz,pgf}
\usepackage{epstopdf}
\usepackage{enumerate}
\usepackage{array}
\usepackage{cite}
\usepackage{booktabs}
\usepackage{ytableau} 
\usepackage{lscape}

\RequirePackage{xcolor}
\definecolor{darkgreen}{rgb}{0,0.4,0}
\definecolor{green}{rgb}{0,0.4,0}
\definecolor{BrickRed}{rgb}{0.65,0.08,0}
\RequirePackage{hyperref}\hypersetup{colorlinks=true,linkcolor=blue,citecolor=darkgreen,filecolor=BrickRed,urlcolor=darkgreen}

\newcommand{\LandauO}{\mathcal{O}}
\newcommand{\Ai}{\text{\normalfont Ai}}

\newcommand{\oeis}[1]{\href{http://oeis.org/#1}{#1}}

\title
{Asymptotics of relaxed \texorpdfstring{$k$}{k}-ary trees}

\author{Manosij Ghosh Dastidar}
{Institut f{\"u}r Diskrete Mathematik und Geometrie, Technische Universit\"at Wien, Wien, Austria}
{}
{https://orcid.org/0000-0003-0721-4979}
{supported by the Austrian Science Fund (FWF):~P~34142.}

\author{Michael Wallner}
{Institut f{\"u}r Diskrete Mathematik und Geometrie, Technische Universit\"at Wien, Wien, Austria
\and \url{https://dmg.tuwien.ac.at/mwallner/}}
{}
{https://orcid.org/0000-0001-8581-449X}
{supported by the Austrian Science Fund (FWF):~P~34142.}

\authorrunning{M.\ Ghosh Dastidar and M.\ Wallner} 

\Copyright{Manosij Ghosh Dastidar and Michael Wallner} 



\EventEditors{}
\EventNoEds{1}
\EventLongTitle{}
\EventShortTitle{}
\EventAcronym{}
\EventYear{}
\EventDate{}
\EventLocation{}
\EventLogo{}
\SeriesVolume{}
\ArticleNo{1}

\begin{document}

\maketitle

\begin{abstract}
A relaxed $k$-ary tree is an ordered directed acyclic graph with a unique source and sink in which every node has out-degree $k$.
These objects arise in the compression of trees in which some repeated subtrees are factored and repeated appearances are replaced by pointers.
We prove an asymptotic theta-result for the number of relaxed $k$-ary tree with $n$ nodes for $n \to \infty$. 
This generalizes the previously proved binary case to arbitrary finite arity, and shows that the seldom observed phenomenon of a stretched exponential term $e^{c n^{1/3}}$ appears in all these cases.
We also derive the recurrences for compacted $k$-ary trees in which all subtrees are unique and minimal deterministic finite automata accepting a finite language over a finite alphabet.
\thispagestyle{empty}
\end{abstract}


\section{Introduction and Main Result}

The enumeration of directed acyclic graphs (DAGs) is an important and timely topic in computer science~\cite{bousquet2015xml}, mathematics~\cite{vdHofstad2017Graphs,Bollobas2001Graphs,JansonEtal2000Graphs}, and many related areas such as phylogenetics~\cite{McDiarmidSempleWelsh2015Networks} and theoretical physics~\cite{GuttmannWhittington1978Graphs,guttmann2015analysis}.
Several problems have remained open for a long time, with bounds sometimes differing by an exponential factor.
One of those problems is the enumeration of minimal deterministic finite automata (DFAs) with $n$ states recognizing a finite language over a finite alphabet~\cite{Liskovets2006exact}. 
In~\cite{ElveyPriceFangWallner2020DFA} Elvey Price, Fang, and Wallner solved the corresponding asymptotic counting problem for a binary alphabet, and uncovered the remarkable phenomenon of a stretched exponential term $e^{c \, n^{1/3}}$. 
This term provides an explanation for the previously encountered difficulties.
For example, the associated generating function cannot be algebraic, and it can only be D-finite (satisfy a linear differential equation with polynomial coefficients) if it has an irregular singularity. 

This phenomenon was first observed by the above-mentioned authors in~\cite{ElveyPriceEtal2021Compacted} in the asymptotic number of another class of DAGs: compacted binary trees of size $n$. 
These arise in the compression of XML documents~\cite{bousquet2015xml}, in the common subexpression problem in, e.g., compiling~\cite{DowneySethiTarjan1980variations}, and in data structures of, e.g., computer algebra software~\cite{flss90}. 
Since then, this phenomenon has additionally been shown in many classes of DAGs and related objects, such as
phylogenetic networks~\cite{Changetal2023dcombining}, 
permutation patterns~\cite{Wallner2023Inversion},
and
Young tableaux~\cite{BanderierWallner2021Walls}.
In this paper, we show that the examples of compacted binary trees and minimal DFAs are just single cases of infinite families admitting a stretched exponential. 
Our main result is the following asymptotics of a super-class of the latter two.
\begin{theorem}
Let $k\geq2$ be an integer. The number of relaxed $k$-ary trees satisfies for $n \to \infty$
\label{theo:relaxedasymptotics}
\begin{align*}
    \Theta \left( n!^{k-1} \left(\frac{k^k}{(k-1)^{k-1}}\right)^n e^{3 \left(\frac{k(k-1)}{2}\right)^{1/3} a_1 n^{1/3}} n^{\frac{2k-1}{3}} \right),
\end{align*}
where $a_1 \approx -2.338$ is the largest root of the Airy function $\Ai(x)$ defined as the unique function satisfying $\Ai''(x)=x \Ai(x)$ and $\lim_{x \to \infty} \Ai(x)=0$.
\end{theorem}

\pagebreak

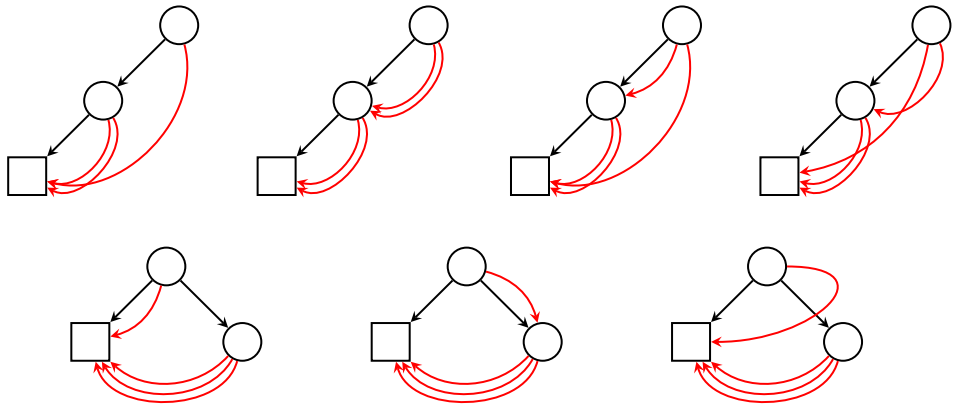
\begin{figure}[ht!]
    \centering

    \tikzstyle{mycircle}=[circle, draw, minimum size=0.5cm]
    
    \begin{tikzpicture}[->,>=stealth,line width=0.75pt,level/.style={sibling distance = 2.5cm/#1, level distance = 1.5cm}] 
    \node[mycircle] (node1-1) at (0,0) {};
    \node[mycircle] (node1-2) at (-1,-1) {};
    \node[rectangle, draw, minimum size=0.5cm] (sink1) at (-2,-2) {};
    \draw (node1-1) -- (node1-2);
    \draw (node1-2) -- (sink1);
    \draw[red] (node1-1) to[bend left=60] (sink1);
    \draw[red] (node1-2) to[bend left=60] (sink1);
    \draw[->, red] (node1-2) to[bend left=75] (sink1);
    \end{tikzpicture}
    \hspace{0.5cm}
    \begin{tikzpicture}[->,>=stealth,line width=0.75pt,level/.style={sibling distance = 2.5cm/#1, level distance = 1.5cm}] 
    \node[mycircle] (node2-1) at (0,0) {};
    \node[mycircle] (node2-2) at (-1,-1) {};
    \node[rectangle, draw, minimum size=0.5cm] (sink2) at (-2,-2) {};
    \draw (node2-1) -- (node2-2);
    \draw (node2-2) -- (sink2);
    \draw[->, red] (node2-1) to[bend left=60] (node2-2);
    \draw[->, red] (node2-1) to[bend left=75] (node2-2);
    \draw[->, red] (node2-2) to[bend left=60] (sink2);
    \draw[->, red] (node2-2) to[bend left=75] (sink2);
    \end{tikzpicture}
    \hspace{0.5cm}
    \begin{tikzpicture}[->,>=stealth,line width=0.75pt,level/.style={sibling distance = 2.5cm/#1, level distance = 1.5cm}] 
    \node[mycircle] (node3-1) at (0,0) {};
    \node[mycircle] (node3-2) at (-1,-1) {};
    \node[rectangle, draw, minimum size=0.5cm] (sink3) at (-2,-2) {};
    \draw (node3-1) -- (node3-2);
    \draw (node3-2) -- (sink3);
    \draw[->, red] (node3-1) to[bend left=30] (node3-2);
    \draw[->, red] (node3-1) to[bend left=60] (sink3); 
    \draw[->, red] (node3-2) to[bend left=60] (sink3);
    \draw[->, red] (node3-2) to[bend left=75] (sink3);
    \end{tikzpicture}
    \hspace{0.5cm}
    \begin{tikzpicture}[->,>=stealth,line width=0.75pt,level/.style={sibling distance = 2.5cm/#1, level distance = 1.5cm}] 
    \node[mycircle] (node4-1) at (0,0) {};
    \node[mycircle] (node4-2) at (-1,-1) {};
    \node[rectangle, draw, minimum size=0.5cm] (sink4) at (-2,-2) {};
    \draw (node4-1) -- (node4-2);
    \draw (node4-2) -- (sink4);
    \draw[->, red] (node4-1) to[bend left=35] (sink4); 
    \draw[->, red] (node4-1) to[bend left=70] (node4-2); 
    \draw[->, red] (node4-2) to[bend left=60] (sink4);
    \draw[->, red] (node4-2) to[bend left=75] (sink4);
    \end{tikzpicture}

    \vspace{5mm}

    \begin{tikzpicture}[->,>=stealth,line width=0.75pt,level/.style={sibling distance = 2.5cm/#1, level distance = 1.5cm}]
    \node[mycircle] (top1) at (0,0) {};
    \node[mycircle] (right1) at (1,-1) {};
    \node[rectangle, draw, minimum size=0.5cm] (sink1) at (-1,-1) {};
    \draw (top1) -- (sink1);
    \draw (top1) -- (right1);
    \draw[->, red] (top1) to[bend left=30] (sink1);
    \draw[->, red] (right1) to[bend left=45] (sink1);
    \draw[->, red] (right1) to[bend left=60] (sink1);
    \draw[->, red] (right1) to[bend left=75] (sink1);
    \end{tikzpicture}
    \hspace{1.25cm}
    \begin{tikzpicture}[->,>=stealth,line width=0.75pt,level/.style={sibling distance = 2.5cm/#1, level distance = 1.5cm}]
    \node[mycircle] (top2) at (0,0) {};
    \node[mycircle] (right2) at (1,-1) {};
    \node[rectangle, draw, minimum size=0.5cm] (sink2) at (-1,-1) {};
    \draw (top2) -- (sink2);
    \draw (top2) -- (right2);
    \draw[->, red] (top2) to[bend left=30] (right2);
    \draw[->, red] (right2) to[bend left=45] (sink2);
    \draw[->, red] (right2) to[bend left=60] (sink2);
    \draw[->, red] (right2) to[bend left=75] (sink2);
    \end{tikzpicture}
    \hspace{1.25cm}
    \begin{tikzpicture}[->,>=stealth,line width=0.75pt,level/.style={sibling distance = 2.5cm/#1, level distance = 1.5cm}]
    \node[mycircle] (top3) at (0,0) {};
    \node[mycircle] (right3) at (1,-1) {};
    \node[rectangle, draw, minimum size=0.5cm] (sink3) at (-1,-1) {};
    \draw (top3) -- (sink3);
    \draw (top3) -- (right3);
    \draw[->, red] (top3) to[out=0,in=0,looseness=2.5] (sink3); 
    
    \draw[->, red] (right3) to[bend left=45] (sink3);
    \draw[->, red] (right3) to[bend left=60] (sink3);
    \draw[->, red] (right3) to[bend left=75] (sink3);
    \end{tikzpicture}   
    
    \caption{All $7$ ternary relaxed ternary trees with $2$ internal nodes.}
    \label{fig:ternrayn2}
\end{figure}

\section{Bijections and recurrences}

In this paper, we consider a special class of DAGs, in which the outgoing edges (equivalently, the children) are equipped with an order.

\begin{definition}
  An \emph{ordered DAG} is a directed acyclic graph where there a left to right ordering among the children.   
\end{definition}

This brings us to the main object of this paper: the class of relaxed trees, a subclass of ordered DAGs, defined as follows; see Figure~\ref{fig:ternrayn2} and \ref{fig:relaxedternary}.
The word relaxed signifies that they are a super-class of compacted trees, which are in bijection with trees after a compression procedure; see~\cite{flss90,ElveyPriceEtal2021Compacted}.

\begin{definition}
A \emph{relaxed $k$-ary tree} is an ordered DAG consisting of a unique source and a unique sink such that every node except the sink has out degree exactly $k$. 
  Its \emph{spine} is the spanning subtree created by the depth first search.
  The edges of the spine are called \emph{internal edges}, 
  while the other edges are called \emph{external edges} or \emph{pointers}.
  All nodes except the unique sink, are called \emph{internal nodes}.
\end{definition}

\begin{figure}[ht!]
    \centering

\begin{tikzpicture}[->,>=stealth,level/.style={sibling distance = 4cm/#1,
  level distance = 1.3cm}] 
\node [circle,draw] (root){8}
  child{node [circle,draw] (four) {4}
    child{node [rectangle,draw,minimum size=0.5cm] (one) {1}}
    child{node [circle,draw] (two) {2}}
    child{node [circle,draw] (three) {3}}
  }
  child{node [circle,draw] (six) {6}
    child{node [circle,draw] (five) {5}}
  }
  child{node [circle,draw] (seven) {7}
  };

\draw[->, red] (two) to[bend left] (one);
\draw[->, red] (two) to[bend left=40] (one);
\draw[->, red] (two) to[bend left=60] (one);
\draw[->, red] (three) to[bend left] (one); 
\draw[->, red] (three) to[bend left=40] (two); 
\draw[->, red] (three) to[bend left=60] (two); 
\draw[->, red] (five) to[bend left] (one);
\draw[->, red] (five) to[bend left=60] (two); 
\draw[->, red] (five) to[bend left=50] (three);
\draw[->, red] (six) to[bend left] (three);
\draw[->, red] (six) to[bend left] (four);
\draw[->, red] (seven) to[bend left] (six);
\draw[->, red] (seven) to[bend left] (five);
\draw[->, red] (seven) to[bend left=60] (five);

\end{tikzpicture}
\caption{Example of relaxed ternary tree with 7 internal nodes (circles) labelled in postorder. The unique sink is depicted by a square. The black edges belong to the spine, the red ones are so-called pointers.}
\label{fig:relaxedternary}
\end{figure}
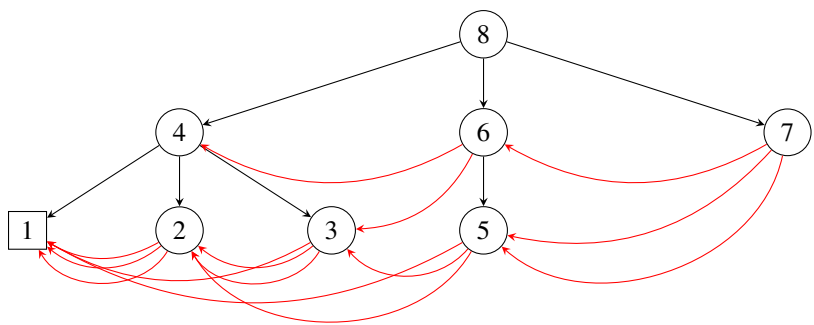

We start as in \cite{ElveyPriceEtal2021Compacted} by drawing a bijection between relaxed $k$-ary trees and Dyck paths with weights on their horizontal steps. This is done so that we can convert internal nodes into vertical steps and pointers into horizontal steps.

Every internal node $u$ in the spine of a relaxed tree has a $k$-ary tree $T(u)$ associated with it where the nodes are traversed in postorder. 
\begin{definition}
    A \emph{compacted} $k$-ary tree is a special case of a relaxed $k$-ary tree where for any arbitrary two nodes $u,v$ in the spine, the $k$-ary trees associated with them $T(u), T(v)$ are not identical. 
\end{definition}

In order to count relaxed $k$-ary trees, we will now describe a bijeciton to a class of paths, which are easier to enumerate. Let us first define the specific paths.

\begin{definition}
A horizontally \( k \)-decorated path \( P \) is defined as a lattice path consisting of up steps $U=(0,1)$ and horizontal steps $H=(1,0)$ from \( (0, -1) \)  with decorations such that: 
\begin{itemize}
    \item The first step is an U step, and its removal leaves a path never crossing the diagonal $y=\frac{x}{k-1}$.
    \item Below each H-step, there is exactly one cross in one of the unit boxes below this H-step and $y=-1$. 
\end{itemize}
\end{definition}

\begin{figure}[t]
    \centering
    \includegraphics[width=0.95\textwidth]{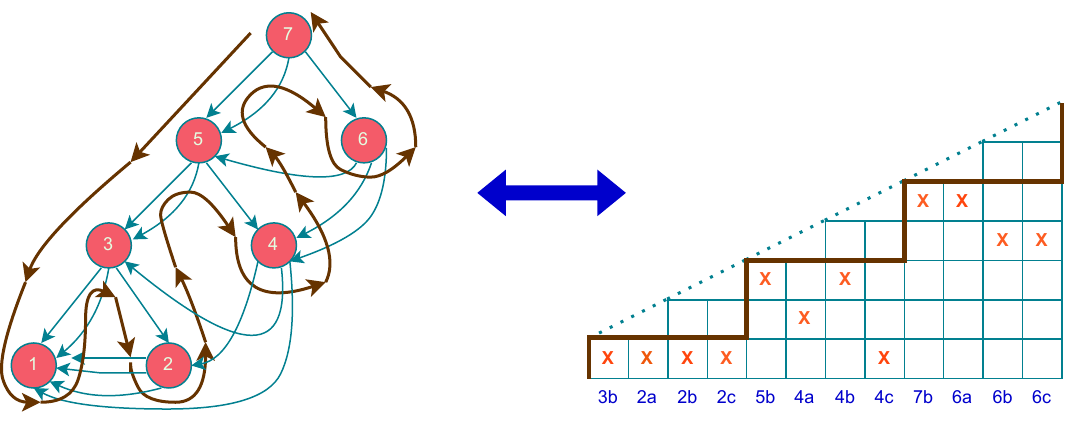}
    \caption{Example of the bijection described in Lemma~\ref{lem:bijectionrelaxed} between relaxed ternary trees and horizontally $3$-decorated paths. }
    \label{fig:bijection}
\end{figure}

The following lemma describes the bijection between relaxed $k$-ary trees and horizontally $k$-decorated paths.
An example is shown in Figure~\ref{fig:bijection}.

\begin{lemma}
    \label{lem:bijectionrelaxed}
    There is a bijection between horizontally $k$-decorated paths ending at $((k-1)n,n)$ and relaxed $k$-ary trees of size $n+1$.    
\end{lemma}

\begin{proof}
Let $R$ be a relaxed $k$-ary tree of size $n+1$.
We transform this tree uniquely into a $k$-decorated path of size ending at~$((k-1)n,n)$.  
We traverse the spine of $R$ in post-order and label nodes from $1$ to $n+1$. 
During the traversal when we move up (i.e., traverse a spine edge the second time) we add an $U$ step 
and when crossing a pointer for the last time (i.e., circling around its parent) we add an $H$ step. 
Thereby, we associate each U step with the node we leave, and each H step with the parent of its pointer.
Moreover, below each H step we draw a cross in the unit box that intersects the column of this H step with the row of the U step that is associated with the target of the pointer of this H step.

Observe now, that the first step is always an U step, as the unique sink is processed first in post-order and has no pointers. 
Thus, after removing this step, we start at the origin $(0,0)$.
Furthermore, note that the $R$ consists of $n$ spine edges, and 
$(k-1)n$ pointers. 
Hence, the path consists of $n$ U steps and $(k-1)n$ H steps, and after attaching a final $U$ step, the path ends at $((k-1)n,n)$.
It remains to show that the path never crosses $y=\frac{x}{k-1}$. 
Note that before an $U$ step is added, the relaxed subgraph of which the associated node is the root has been traversed. 
The subgraph consisting only of the node $1$ is treated by the fact that we start at $(0,-1)$.
Now we proceed by induction on the depth of the subgraph. 
We need to distinguish, whether the node is on the left-branch from the root or not.
First, we assume it is not. 
The minimal cases is a single pointer, which leads to a step $H$ and increases the distance to the diagonal $y=\frac{x}{k-1}$ by one unit. 
Now assume that a subgraph with $i>0$ nodes is given.
The root has $k$ children, which all have size less than $i$. 
Thus by induction, each part does not cross $y=\frac{x}{k-1}$ and moves one unit further to the right from the diagonal.
Therefore, processing the $k$ children moves the path $k$ steps to the right from $\frac{x}{k-1}$, while (after that) processing the root, moves the path one unit up.
Thus the path stays below $y=\frac{x}{k-1}$ and is one unit further to the right.
Second, we assume that the node is on the diagonal. 
Then, the first child does not move the path one unit to the right, but the other $k-1$ do. 
By induction, the path still does not cross the diagonal $\frac{x}{k-1}$ but the distance now also stays the same.

Now in the reverse direction, let us take a horizontally k-decorated path ending at $((k-1)n,n)$. We want to recover a relaxed $k$-ary tree from this path.
We start by noting that the number of up steps in the path is $n$ and therefore the number of nodes is $n+1$. Our first step is a horizontal step.

We start by adding an H step which corresponds to the left most leaf labelled 1. Then along with the path P, we create the spine of the the relaxed $k$-ary tree. Thus along the up steps we create nodes of the spine as the $i$th up step creates the $(i+1)$th node in post-order. Along the horizontal steps we say that we add a pointer from the existing node that we are in to a node $j+1$ which is indicated by the cross placed on the level $j$ below the path (counting from the bottom up). Completing this process we get a relaxed $k$-ary tree with its nodes labelled in post-order. 
\end{proof}

This bijection allows directly to derive the following recurrence relations, following the step-by-step construction of the paths.
In Table~\ref{tab:relaxed} we give the initial terms for $k=2,\dots,5$.
   
\begin{proposition}
    Let $r_{n,m}$ be the number of horizontally decorated paths ending at $(n,m)$. 
\begin{align}
    \label{eq:relaxedNEsteps}
    r_{n,m} &= r_{n,m-1} + (m+1)r_{n-1,m}, 
             && \text{for } 1 \leq m \leq \frac{n}{k-1},\\             
    r_{n,m} &= 0  && \text{for } m > \frac{n}{k-1}, \notag\\
    r_{n,0} &= 1  && \text{for } n \geq 0. \notag
\end{align}
Thus, by Lemma~\ref{lem:bijectionrelaxed}, the number of relaxed $k$-ary trees with $n$ internal nodes is equal to $r_{(k-1)n,n}$.
\end{proposition}

\begin{table}[ht!]
\centering
\caption{Number of relaxed $k$-ary trees. 
Also the number of unlabeled acyclic single-source automata with $n$ transient states on a $k$-letter input alphabet. 
The matrix consisting of these rows is given by~\oeis{A128249}}
\label{tab:relaxed}
\begin{tabular}{@{}llc@{}}
\toprule
$k$ & Relaxed $k$-ary trees ($r_k$) & OEIS \\ 
\midrule
2 & $(	1, 3, 16, 127, 1363, 18628, 311250, 6173791,   \ldots)$ & \oeis{A082161} \\
3 & $(1, 1, 7, 139, 5711, 408354, 45605881, 7390305396, \ldots)$ & \oeis{A082162} \\
4 & $(1, 1, 15, 1000, 189035, 79278446, 63263422646, 86493299281972)$ & \oeis{A102102} \\
5 & $(1, 1, 31, 6631, 5470431, 12703473581, 68149976969707, \ldots)$ & --- \\ 
\bottomrule
\end{tabular}
\end{table}

In a next step, we will also state the respective recurrences for compacted $k$-ary trees and minimal DFAs accepting a accepting a finite language on an alphabet of size $k$. 
These results generalize~\cite[Proposition~2.11]{ElveyPriceEtal2021Compacted} and \cite[Proposition~5]{ElveyPriceFangWallner2020DFA}, respectively.

In a relaxed $k$-ary tree an internal node will be called a \emph{cherry} if all its children are pointers.
\begin{theorem}
  A relaxed $k$-ary tree C is a compacted $k$-ary tree iff no internal nodes $u$, $v$ have the same children in the same order. Moreover, if C is not a compacted $k$-ary tree there exists a pair $(u,v)$ with identical children and $v$ is a cherry and $u$ precedes $v$ in postorder. 
\end{theorem}
    
\begin{proof}
Let us take C to be a compacted $k$-ary tree. Let two internal nodes $u$, $v$ have the same children in the same order. This implies that the $k$-ary trees associated with $u$ and $v$ are isomorphic, which violates the definition of a compacted tree. 
So conversely, if C is not a compacted tree then we can find at least one pair of internal nodes $(u,v)$ such that the $k$-ary tree associated with them will be the same. At this point if $v$ is a cherry then we are done. But if $v$ is not a cherry then we take the first child of both nodes $T(u_1)$ and $T(v_1)$, respectively, and note that these have to be identical as well. Now if $v_1$ is not a cherry then we can continue this process. Therefore by infinite descent we see that there has to be two nodes $u, v$ such that $v$ is a cherry and $u$ precedes $v$ in postorder. 
\end{proof}

 In order to continue, we need the following concept.
 First, recall that in each horizontally $k$-decorate path, each step $H$ is decorated by a marked box below. 
 Further, each vertical step $V$ corresponds to an internal node in the $k$-ary tree. 
 Thus, we assign to each step $V$ a $k$-tuple $(v_1,v_2,\ldots, v_k)$ corresponding to the nodes of its children. (This will help us to construct the compacted trees from the relaxed ones.) 
Let us define a \emph{$C_k$-decorated path} as a horizontally decorated path with the restriction that for consecutive steps $H^k V$ ($k$ steps $H$ followed by $V$) the $k$-tuple $(h_1,h_2,\ldots h_k)$ corresponding to the labels below the $H$-steps, we have $(h_1,h_2,\ldots h_k)\neq (v_1,v_2, \ldots v_k)$ for all preceding steps $V$. 

\medskip
\begin{proposition}
There is a bijection between the number of compacted $k$-ary trees of size $n$ and $C_k$-decorated paths of length $kn$.
\end{proposition}

\begin{proof}
Note that the $C_k$-decorated paths are a subset of the horizontally $k$-decorated paths and the previous bijection sends the relaxed $k$-ary trees to horizontally $k$-decorated paths. The $C_k$-decorated paths have been constructed in this way to reinforce the condition that the relaxed trees corresponding to the these decorated paths are such that the $k$-ary trees associated with any two arbitrary nodes can never be identical. 
\end{proof}

This allows us, again by a direct step-by-step construction, to derive the following bivariate recurrence for $C_k$-decorated paths; see Table~\ref{tab:compacted}.

\begin{proposition}
\label{prop:compacted}
Let $c_{n,m}$ denote the number of $C_k$-decorated paths ending at $(n,m)$. Then 
\begin{align*}
c_{n,m} &= c_{n,m-1} + (m+1)c_{n-1,m} - (m-1)c_{n-k,m-1} && \text{for } 1 \leq m \leq \frac{n}{k-1},  \\
c_{n,m} &= 0  && \text{for } m > \frac{n}{k-1}, \\
c_{n,0} &= 1  && \text{for } n \geq 0.
\end{align*}
The number of compacted $k$-ary trees with $n$ internal nodes is $c_{(k-1)n,n}$.
\end{proposition}

\begin{table}[ht!]
\centering
\caption{Number of compacted $k$-ary trees, which are defined as relaxed $k$-ary trees with the additional constraint that each fringe subtree is unique.}
\label{tab:compacted}
\begin{tabular}{@{}llc@{}}
\toprule
$k$ & Compacted $k$-ary trees ($c_k$) & OEIS \\ 
\midrule
2 & $(1, 1, 3, 15, 111, 1119, 14487, 230943, 4395855, \ldots)$ & \oeis{A254789} \\
3 & $(1, 1, 7, 133, 5299, 371329, 40898599, 6561293893, \ldots)$ & --- \\
4 & $(1, 1, 15, 975, 182175, 75961695, 60422966655, 82450320955455, \ldots)$ & --- \\
5 & $(1, 1, 31, 6541, 5373571, 12458850121, 66790559866471, \ldots)$ & --- \\ 
\bottomrule
\end{tabular}
\end{table}

We have a similar recursion for the minimal DFAs over a $k$-ary alphabet. We use the same scheme for the bijection with the $H$ steps having a decoration. Since we have an accepting and rejecting state the $V$ steps are colored either white or green. 
Again, we summarize the first values of Table~\ref{tab:minimal}.

\begin{proposition}
\label{prop:minimalDFAs}
Let $b_{n,m}$ denote the number of automatic $B$-paths corresponding to DFAs ending in $(n,m)$. 
Then
\begin{align*}
b_{n,m} &= 2b_{n,m-1} + (m+1)b_{n-1,m} - m \, b_{n-k,m-1} && \text{for } 1 \leq m \leq \frac{n}{k-1},  \\
b_{n,m} &= 0  && \text{for } m > \frac{n}{k-1}, \\
b_{n,0} &= 1  && \text{for } n \geq 0.
\end{align*}
The number of minimal DFAs accepting a accepting a finite language on an alphabet of size $k$ with $n+1$ states is $b_{(k-1)n,n}$.
\end{proposition}

\begin{table}[ht!]
\centering
\caption{Number of minimal deterministic finite automata recognizing a finite $k$-letter language.}
\label{tab:minimal}
\begin{tabular}{@{}llc@{}}
\toprule
$k$ & DFA ($m_k$) & OEIS \\ 
\midrule
2 & $(1, 1, 6, 60, 900, 18480, 487560, \ldots)$ & \oeis{A331120} \\
3 & $(1,1,14,532,42644, 6011320, 1330452032\ldots)$ & ---  \\
4 & $(1,1,30,3900,1460700, 1220162880\ldots)$ & ---  \\
5 & $(1,1,62,26164,43023908, 199596500056\ldots)$ & ---  \\ 
\bottomrule
\end{tabular}
\end{table}

\newcommand{\dd}{d}

\section{Asymptotics of relaxed \texorpdfstring{$k$}{k}-ary trees}

The goal of this section is to prove our main Theorem~\ref{theo:relaxedasymptotics} on relaxed $k$-ary trees.
As the proof is rather technical and complex, we will begin with an overview of the main steps.

First, in Section~\ref{sec:rectrafo} we transform the recurrence~\eqref{eq:relaxedNEsteps} into a new recurrence $(d_{i,j})_{i,j\geq 0}$ by changing the used steps and extracting the super-exponential and certain exponential and polynomial contributions. 
Second, in Section~\ref{sec:heuristic} we perform a heuristic analysis of the asymptotics of $(d_{i,j})_{i,j\geq 0}$, and guess the shape for a rigorous proof.
Using this shape, we build and prove in Section~\ref{sec:explicitbounds} explicit sequences $(\tilde{X}_{i,j})_{i,j\geq0}$ and $(\hat{X}_{i,j})_{i,j\geq0}$ that satisfy the same recurrence as $(d_{i,j})_{i,j\geq 0}$ but with the equality sign replaced by the inequality signs $\leq$ and $\geq$, respectively.
This is the most technical part in which we heavily rely on computer algebra.
Finally, in Section~\ref{sec:proofmainresult} we use these explicit sequences to prove inductively asymptotically tight upper and lower bounds, which yield our main theorem.

\subsection{Transformation into a Dyck-like recurrence}
\label{sec:rectrafo}

We start from recurrence~\eqref{eq:relaxedNEsteps} and our goal is to determine the asymptotics of $r_{(k-1)n,n}$. 
First, we observe that the path closest to the diagonal given by $(H^{k-1}V)^n$ has a weight $(n!)^{k-1}$ as there are $n$ $H$-steps at height $0,1,\dots,n-1$. 
All other paths will get smaller weights. 
Therefore, after rescaling with this weight, all paths have a weight bounded by one. 
Note that it is technically easier to work with the rescaling $((k-1)n)!$, which has by Stirling's formula the same super-exponential growth.
The difference in the exponential growth is a factor $(k-1)^{(k-1)n}$, which we will take into account later in~\eqref{eq:rkndijrelaxed}.
Thus, we set $\tilde{r}_{n,m} = \frac{r_{n,m}}{((k-1)n)!}$ and get the new recurrence
\begin{align*}
    \tilde{r}_{n,m} &= \frac{m+1}{n} \tilde{r}_{n-1,m} + \tilde{r}_{n,m-1}, 
             & m &\leq \frac{n}{k-1}.
\end{align*}
Next, as we are interested in the asymptotics of $\tilde{r}_{(k-1)n,n}$, we will transform this recurrence from North and East steps, to Dyck-like up and down steps. 
For this purpose, we define the new variables
\begin{align*}
    \begin{cases}
        i = n+m, \\
        j = n-(k-1)m,
    \end{cases}
    \quad \text{ equivalently } \qquad
    \begin{cases}
       n = \frac{(k-1)i + j}{k}, \\
       m = \frac{i-j}{k}.
    \end{cases}
\end{align*}
The idea behind this choice, is that $i$ tracks the length of the walk, i.e., the number of steps, and $j$ the distance to the diagonal. By this choice, we have $i,j \geq 0$. 
This gives the following generalized Dyck-like recurrence:
\begin{align}
    \tilde{\dd}_{i,j} &= \tilde{U}(i,j) \tilde{\dd}_{i-1,j-1} + \tilde{\dd}_{i-1,j+k-1}, & i >0, j &\geq 0,
\end{align}
with the initial condition $\tilde{\dd}_{0,0}=1$ and the weight $\tilde{U}(i,j) = \frac{i-j+k}{(k-1)i+j}$.
Thus, we have the new steps $(1,1)$ and $(1,-k+1)$. 
This is a simple directed lattice path model with space-dependent weights. 
Note that as the change in the $x$-direction is one unit per step, a path of length $n$ consists of $n$ steps, and it suffices therefore to track the current altitude.
Therefore, from now on we consider only the changes in the $y$-direction, which we call jumps $+1$ and $-(k-1)$.

\medskip

Let us now look at the drift of this new model.
For now, we assume that $i$ is large and that $j=o(i)$.
The \emph{drift} $\tilde{\delta}(i,j)$ at a point $(i,j)$ is defined as the expected next jump size when leaving $(i,j)$. 
Therefore, at $(i,j)$ the jump $+1$ gets weight $\tilde{U}(i+1,j+1)$ and the jump $-(k-1)$ weight $1$ for $j\geq k-1$ and $0$ otherwise. 
Hence, we get that $\tilde{\delta}(i,j) =1 $ for $0 \leq j \leq k-2$ and
\begin{align*}
    \tilde{\delta}(i,j) &= 
    \frac{i-j+k}{(k-1)(i+1) + j+1} - (k-1) 
     = -\frac{k(k-2)}{k-1} + \LandauO\left(\frac{1}{i}\right),
        & \text{ for }~ j > k-2.
\end{align*}
Note that in the binary case ($k=2$), the first term is zero and therefore the drift is converging to zero for large $i$. 
However, in the general case the drift is negative for large $i$. 
But we can define the following transformation to achieve the same behavior:
$\dd_{i,j} = (k-1)^{2n}\tilde{\dd}_{i,j}$ where $n=\frac{(k-1)i + j}{k}$.
This gives the final Dyck-like recurrence
\begin{align}
    \label{eq:Dyckrecurelaxed}
    \dd_{i,j} &= U(i,j) \dd_{i-1,j-1} + \dd_{i-1,j+k-1}, & i >0, j &\geq 0,
\end{align}
with the initial condition $\dd_{0,0}=1$ and the following weight for the up step
\begin{align}
    \label{eq:Uijrelaxed}
    U(i,j) = \frac{ (k-1)^2(i - j + k)}{(k-1) i + j}%
    =(k - 1) \left(1-\frac{k(j - k+1)}{(k-1) i + j} \right).
\end{align}
The drift $\delta(i,j)$ in this model is again $\delta(i,j) =1 $ for $0 \leq j \leq k-2$ but now we have
\begin{align*}
    \delta(i,j) &= 
        -\frac{k(k-1)(j-k+2)}{k(i+1) -i+j} = -\frac{k(j-k+2)}{i} + \LandauO\left(\frac{1}{i^2}\right),
        & \text{ for }~ j > k-2,
\end{align*}
and therefore converges to zero.
In this final model~\eqref{eq:Dyckrecurelaxed}, we are interested in $d_{kn,0}$ since 
\begin{align}
    \label{eq:rkndijrelaxed}
    \begin{aligned}
    r_{(k-1)n,n}
    &= ((k-1)n)! \tilde{r}_{(k-1)n,n}
    = ((k-1)n)! \tilde{\dd}_{kn,0}
    = \frac{((k-1)n)!}{(k-1)^{2(k-1)n}}\dd_{kn,0}\\
    &\sim \sqrt{k-1} (2\pi)^{1-k/2} \frac{(n!)^{k-1}}{(k-1)^{(k-1)n}} n^{1-k/2} \dd_{kn,0},
    \end{aligned}
\end{align}
where the last equality follows by Stirling's formula.

\subsection{Heuristic analysis}
\label{sec:heuristic}

\newcommand{\bb}{B}
\newcommand{\aiarg}{\alpha}

In the next step, we will now heuristically analyze the recurrence~\eqref{eq:Dyckrecurelaxed}. 
Inspired by the binary case and numerical experiments, we use the ansatz
\begin{align}
    \label{eq:ansatzone}
    d_{i,j} = h(i) f\left( \frac{j+1}{i^{1/3}} \right).
\end{align}
The function $h(i)$ captures a macroscopic amplitude that is independent of $j$, while the function $f(x)$ captures the local behavior around the origin.
The rescaling by $i^{1/3}$ is motivated by an analogy to pushed Dyck path and the binary case~\cite{ElveyPriceEtal2021Compacted}.

Substituting this ansatz into the recurrence~\eqref{eq:Dyckrecurelaxed} and reordering we get
\begin{align*}
    \frac{h(i)}{h(i-1)} &= \frac{U(i,j) f\left( \frac{j}{(i-1)^{1/3}} \right) + f\left( \frac{j+k}{(i-1)^{1/3}} \right)}{f\left( \frac{j+1}{i^{1/3}} \right)}.
\end{align*}
As the left-hand side is independent of $j$, for large $i$ both sides should have an expansion in $i$ with whose coefficients are independent of $j$. 
Let us now zoom into the region $i^{1/3}$ by binding the variables $i$ and $j$ using the transformation $j=x \, i^{1/3}+1$. 
This gives
\begin{align}
\label{eq:hquot}
    \frac{h(i)}{h(i-1)} &= \frac{U(i,x \, i^{1/3}+1) f\left( \frac{x \, i^{1/3}+1}{(i-1)^{1/3}} \right) + f\left( \frac{x \, i^{1/3}+k+1}{(i-1)^{1/3}} \right)}{f(x)}.
\end{align}
Assuming now that $f(x)$ has a convergent Taylor series expansion around $x$, we get
\begin{align*}
\frac{h(i)}{h(i-1)} &= 
k + \frac{k}{2}\frac{(k-1)f''(x) - 2 x f(x)}{f(x)} i^{-2/3} + \LandauO(i^{-1}).
\end{align*}
Hence, in order to be consistent, the left-hand side needs to have an expansion in decreasing powers of $i^{1/3}$. In particular, it starts like
\begin{align*}
    \frac{h(i)}{h(i-1)} &= k + c \, i^{-2/3} + \LandauO(i^{-1}),
\end{align*}
for a constant $c \in \mathbb{R}$.
Comparing coefficients, we see that $f(x)$ has to satisfy the differential equation
\begin{align*}
    f''(x) &= \frac{2 (k x + c)}{k(k-1)} f(x).
\end{align*}
This differential equation is, as in the binary case, solved by the Airy functions of the first and second kind. 
Due to the combinatorial nature of the problem, we require $f(x) \geq 0$ for $x \geq 0$ as well as $\lim_{x \to \infty} f(x) = 0$. 
Therefore, we get up to a multiplicative constant the following solution
\begin{align*}
    f(x) &= \Ai \left(\frac{2^{1/3}(k x + c)}{k (k-1)^{1/3}} \right).
\end{align*}
Now, note that due the boundary conditions $d_{n,-1}=0$ we must have $f(0)=0$. Together, with $f(x)\geq 0$ for $x\geq 0$, this means that the argument of the Airy function must evaluate to the largest zero $a_1 \approx -2.338$ of $\Ai(x)$. 
Therefore
\begin{align*}
    c = \frac{k (k-1)^{1/3}}{2^{1/3}}a_1
\end{align*}
and 
\begin{align*}
    f(x) &= \Ai \left(a_1 + \bb \, x  \right), \qquad
    \text{ where } \qquad \bb = \bigg(\frac{2}{k-1}\bigg)^{1/3}.
\end{align*}

In order to capture the polynomial term, we need to use a more general ansatz than~\eqref{eq:ansatzone} that includes a second function $g(x)$ capturing the scale of order $i^{-1/3}$. In particular, we use
\begin{align}
    \label{eq:ansatztwo}
    d_{i,j} = h(i) \left( f\left( \frac{j+1}{i^{1/3}} \right) + \frac{g\left( \frac{j+1}{i^{1/3}} \right)}{i^{1/3}} \right).
\end{align}
Note that the function $g(x)$ will influence the terms starting from order $i^{-1}$ in~\eqref{eq:hquot}.
Analogous computations as performed above for $f(x)$ lead us to the ansatz 
\begin{align*}
    g(x) &= -\left( \frac{(k+2)x^2}{6(k-1)} + \frac{2^{2/3} a_1 (k-2)x}{6(k-1)^{2/3}} \right) f(x).
\end{align*}

\subsection{Explicit bounds}
\label{sec:explicitbounds}

All these heuristic arguments above guide us to the following two results.
These generalize~\cite[Lemmas~4.2 and 4.4]{ElveyPriceEtal2021Compacted}, whose results are recovered by setting $k=2$.
The proofs are analogous to~\cite{ElveyPriceEtal2021Compacted,Changetal2023dcombining}; for the details we refer to the accompanying Maple worksheet.
In particular, this shows that the method developed in~\cite{ElveyPriceEtal2021Compacted} is powerful enough to analyze bivariate recurrences that include parameters.

\begin{lemma}
	\label{lem:RelaxedXLower}
	Let $k \geq 2$ be an integer and $\bb = (\frac{2}{k-1})^{1/3}$. 
    For all $i,j\geq0$ let
	\begin{align*}	
		\tilde{X}_{i,j} &:= \left(
            1  
            - \frac{2^{2/3} a_1 (k-2)}{6(k-1)^{2/3}} \frac{j}{i^{2/3}}
            - \frac{k+2}{6(k-1)} \frac{j^2}{i}
            + \frac{7k-11}{6(k-1)} \frac{j}{i} \right.\\
            &\qquad+ \left. \frac{a_1^2(k-2)^2 \bb^{4}}{72} \frac{j^2}{i^{4/3}}
            + \frac{a_1(k^2-4)^2 \bb^{5}}{72} \frac{j^3}{i^{5/3}}
            \right)
            \Ai\left(a_{1}+\frac{\bb(j+1)}{i^{1/3}}\right), \\
		 \tilde{s}_i &:= k \left(1+\frac{a_1}{\bb i^{2/3}} + \frac{7k-6}{6i} - \frac{1}{i^{7/6}} \right).
	\end{align*}
	Then, for any $\varepsilon>0$, there exists an $\tilde{i}_0$ such that
	\begin{align*}
		\tilde{X}_{i,j}\tilde{s}_{i} \leq U(i,j) \tilde{X}_{i-1,j+1} + \tilde{X}_{i-1,j-1}
	\end{align*}
	for all $i\geq \tilde{i}_0$ and for all $0 \leq j < i^{2/3-\varepsilon}$, where $U(i,j)$ is defined in~\eqref{eq:Uijrelaxed}.
\end{lemma}

\begin{lemma}
	\label{lem:RelaxedXUpper}
	Let $k \geq 2$ be an integer and $\bb = (\frac{2}{k-1})^{1/3}$. 
	Choose $\eta > \frac{(k+2)^2}{72 (k-1)^2}$ fixed and 
	for all $i,j\geq0$ let
	\begin{align*}	
		\hat{X}_{i,j} &:= \left(
            1  
            - \frac{2^{2/3} a_1 (k-2)}{6(k-1)^{2/3}} \frac{j}{i^{2/3}}
            - \frac{k+2}{6(k-1)} \frac{j^2}{i}
            + \frac{7k-11}{6(k-1)} \frac{j}{i} \right.\\
            &\qquad+ \left. \frac{a_1^2(k-2)^2 \bb^{4}}{72} \frac{j^2}{i^{4/3}}
            + \frac{a_1(k^2-4)^2 \bb^{5}}{72} \frac{j^3}{i^{5/3}}
            + \eta \frac{m^4}{n^2}
            \right)
            \Ai\left(a_{1}+\frac{\bb(j+1)}{i^{1/3}}\right), \\
		 \hat{s}_i &:= k \left(1+\frac{a_1}{\bb i^{2/3}} + \frac{7k-6}{6i} + \frac{1}{i^{7/6}} \right).
	\end{align*}
	Then, for any $\varepsilon>0$, there exists an $\hat{i}_0$ such that
	\begin{align*}
		\hat{X}_{i,j}\hat{s}_{i} \geq U(i,j) \hat{X}_{i-1,j+1} + \hat{X}_{i-1,j-1}
	\end{align*}
	for all $i\geq \hat{i}_0$ and for all $0 \leq j < i^{1-\varepsilon}$, where $U(i,j)$ is defined in~\eqref{eq:Uijrelaxed}.
\end{lemma}

Note that in the binary case $k=2$, many terms in the previous Lemmas~\ref{lem:RelaxedXLower} and \ref{lem:RelaxedXUpper} are zero. 
These terms do not affect the final asymptotics, but they are needed for the proof using generalized Newton polygons.
In particular, they allow to set certain points on the convex hull to zero, which then leads to the same behavior as in the binary case. 
For this reason, the technical proofs follow nearly verbatim the binary ones and we omit them in this extended abstract and refer to~\cite{ElveyPriceEtal2021Compacted} and the accompanying Maple worksheet.

These explicit recurrences that satisfy the recurrence~\eqref{eq:Dyckrecurelaxed} with $=$ replaced by $\leq$ and $\geq$, are the key ingredient to prove our main theorem.

\subsection{Proof of Theorem~\ref{theo:relaxedasymptotics} on relaxed \texorpdfstring{$k$}{k}-ary trees}
\label{sec:proofmainresult}

\newcommand{\ddh}{\tilde{\dd}}

We start with the lower bound.     
First, we define a sequence $X_{i,j} : =\max\{\tilde{X}_{i,j},0\}$ which satisfies the inequality of Lemma~\ref{lem:RelaxedXLower} for \emph{all} $j \leq \frac{i}{k-1}$.
Note that the factor of the Airy function $\Ai$ becomes negative for large $j$.
Then, we define an explicit sequence $\tilde{h}_i := \tilde{s}_i \tilde{h}_{i-1}$ for $i >0$ and $\tilde{h}_0 = \tilde{s}_0$.
Using this we can prove by induction that $\dd_{i,j} \geq C_0 \tilde{h}_{i} X_{i,j}$ for some constant $C_0>0$ and all $i \geq \tilde{i}_0$ and all $0 \leq j \leq \frac{i}{k-1}$. 
Therefore,
\begin{align}
   \dd_{kn,0}
   &\geq C_0 \tilde{h}_{kn} X_{kn,0} \notag\\
   &\geq C_0 \prod_{i=1}^{kn} k \left(1+\frac{a_1}{\bb i^{2/3}} + \frac{7k-6}{6i} - \frac{1}{i^{7/6}} \right){\Ai}\left(a_{1}+\frac{\bb}{(kn)^{1/3}}\right) \notag \\
   &\geq C_1 k^{kn} e^{3 a_1 (kn)^{1/3}/\bb} n^{\frac{7k-8}{6}}. \label{eq:dknlowerbound}
\end{align}
Finally, combining this with~\eqref{eq:rkndijrelaxed} we get the lower bound.

We continue with the upper bound, whose proof is similar, yet more technical.
The starting point is Proposition~\ref{lem:RelaxedXUpper} and, as in the lower bound, a function $X_{i,j}$ that is valid for \emph{all} $0 \leq j \leq \frac{i}{k-1}$.
For this purpose we define as in the binary case~\cite{ElveyPriceEtal2021Compacted} a sequence $\ddh_{i,j}$ depending on some large integer parameter $I>0$ such that   
\begin{align}
    \label{eq:defddh}
    \ddh_{i,j} := 
        \begin{cases}
            \dd_{i,j} & \text{ for } 0 \leq j \leq i^{3/4} \text{ and } i>I, \\
            0 & \text{ otherwise.}            
        \end{cases}
\end{align}
The missing key step is now to show that $\dd_{kn,0} =\LandauO(\ddh_{kn,0})$.
Combining this with the analogous computations performed for the lower bound above we get
\begin{align*}
   \ddh_{kn,0}
   &\leq \hat{C}_1 k^{kn} e^{3 a_1 (kn)^{1/3}/\bb} n^{\frac{7k-8}{6}}.
\end{align*}

To complete the prove we show $\dd_{kn,0} \leq 2 \ddh_{kn,0}$ using lattice path theory and computer algebra.
We start from Equation~\eqref{eq:Dyckrecurelaxed} of $\dd_{i,j}$, which we interpret as a recurrence counting lattice paths.
They are composed of steps $(1,1)$ weighted by $U(i,j)$ and $(1,-k+1)$ weighted by $1$ when the respective step ends at $(i,j)$.
The total weight of a path is the product of its weights. 
Now, we are interested in the paths never crossing $y=0$ and ending at $(kn,0)$. 
Let now $p_{\ell,k,kn}$ be the number of such paths starting at $(r,s)$ and ending at $(kn,0)$.
From~\eqref{eq:Dyckrecurelaxed} we directly get
\begin{align*}
   p_{r,s,kn} = 
    \frac{(k-1)^2(r-s+k)}{(k-1)(r+1)+s+1} p_{r+1,s+1,kn}
    + 
    p_{r+1,s-k+1,kn},
\end{align*}
with $p_{\ell,-1,kn}=0$ and $p_{kn,s,kn} = \delta_{s,0}$, where the factor is $U(r+1,s+1)$.

Now, we are able to show that
\begin{align}
    \label{eq:plj2ninequ}
   \frac{p_{r,s_1,kn}}{s_1+1} \geq \frac{p_{r,s_2,kn}}{s_2+1},
\end{align}
for integers $0 \leq s_1 < s_2  \leq r \leq kn$ such that $k \mid s_2-s_1$.
The proof is completely elementary and uses reverse induction on $r$ starting from $r=kn$. We refer to our accompanying Maple worksheet.

Finally, from~\eqref{eq:plj2ninequ} we directly get 
\begin{align}
    \label{eq:p-2x2y-2x0}
   p_{kx,ky,kn} \leq (ky+1) p_{kx,0,kn}.
\end{align}
This allows to prove the following generalization of~\cite[Lemma~4.6]{ElveyPriceEtal2021Compacted} whose proof follows exactly the same lines and we omit it here.
For the following statement, recall that $\dd_{i,j}$ is the weighted number of paths ending at $(i,j)$. 
Let $\ddh_{i,j}$ be the number of these paths such that no intermediate point $(kx,ky)$ on the path satisfies $x>I_{\varepsilon}$ and $y>x^{3/4}$.

\begin{lemma}\label{lem:chooseconst} 
    For all $\varepsilon>0$ there exists a constant $I_{\varepsilon}>0$ that acts as the parameter $I$ in the definition of $\ddh_{i,j}$ in~\eqref{eq:defddh}, such that $\dd_{kn,0}\leq(1+\varepsilon)\ddh_{kn,0}$ for all $n>0$.
\end{lemma}

\begin{proof}
    As in the binary case, we first rewrite the claimed inequality into 
    $1-\frac{\ddh_{kn,0}}{\dd_{kn,0}} \leq \frac{\varepsilon}{1-\varepsilon}$.
    The left-hand side here represents the proportion of walks that pass through at least one point $(kx,ky)$ such that $x > I_{\varepsilon}$ and $y>x^{3/4}$. 
    Let $s_{x,y,n}$ be the proportion of such walks that pass through one fixed such point:
    \[
        s_{x,y,n} = \frac{\dd_{kx,ky} p_{kx,ky,kn}}{\dd_{kn,0}}.
    \]
    The idea is that the sum of these values over all violating points $x,y$ is of course an upper bound for our claim. 
    So we want to prove that this sum is very small for large $x$.
    
    By~\eqref{eq:p-2x2y-2x0} combined with the fact that $s_{x,y,n}\leq1$, we get
    \begin{align*}
        s_{x,y,n} 
        &\leq \frac{(ky+1)p_{kx,0,kn}d_{kx,ky}}{\dd_{kn,0} s_{x,0,n}} = (ky+1)\frac{\dd_{kx,ky}}{\dd_{kx,0}}\\
        &\leq \frac{ky+1}{C_1 k^{kx} e^{3 a_1 (kx)^{1/3}/\bb} x^{\frac{7k-8}{6}}} (k-1)^{(k-1)x+y}\binom{kx}{x-y}
    \end{align*}
    where we also used the lower bound~\eqref{eq:dknlowerbound} for $d_{kx,0}$ and the crude bound $\dd_{kx,ky} \leq (k-1)^{(k-1)x+y}\binom{kx}{x-y}$. 
    The latter holds, as we may bound non-negative paths ending at $(kx,ky)$ by unconstrained paths with weights $(k-1)$ for the up step $(1,1)$ and $1$ for down step $(1,-k+1)$, since $U(i,j) \leq k-1$ for all $i,j \geq 0$. 
    
    Now this last expression is completely explicit, and for large $x$ one can see that it is of order $\Theta(e^{-x^{1/2}})$.
    Now, the proof of the binary case follows verbatim and the claim follows.
\end{proof}

Finally, this proves $\dd_{kn,0} \leq 2 \ddh_{kn,0}$ and ends the proof of  our main Theorem~\ref{theo:relaxedasymptotics}.

\section{Conclusion and Outlook}

The aim of this paper was to show that the method from~\cite{ElveyPriceEtal2021Compacted} developed for the asymptotics of compacted binary trees can be applied to more general recurrences including a parameter. 
Previously, in~\cite{Changetal2023dcombining} we showed how to handle more general weights including a parameter, while in this paper we generalized the used steps. 
Instead of steps $(1,1)$ and $(1,-1)$ of Dyck type we studied a recurrence with larger steps given by $(1,k)$ and $(1,-1)$.
In particular, we proved in Theorem~\ref{theo:relaxedasymptotics} that in the class of relaxed $k$-ary trees the phenomenon of a stretched exponential appears for any integer $k \geq 2$. 

In the long version, we will also prove the generalizations from the binary to the $k$-ary case for compacted trees and minimal deterministic finite automata accepting a finite language on an alphabet of size $k$ for which we showed in Propositions~\ref{prop:compacted} and \ref{prop:minimalDFAs} that their respective recurrences have similar shapes. 
We expect that the phenomenon of a stretched exponential persists also in these cases. 
In general, the presented method can also be applied to more general cases including several parameters in the weights and also recurrences with more than two steps. 

Another research direction is to generalize the previously studied relaxed binary trees of bounded right height~\cite{GenitriniGittenbergerKauersWallner2016} to the $k$-ary case.
In the binary case, this class has been proved to be D-finite using exponential generating function methods. 
In the $k$-ary case, a similar rescaling, as we used in Section~\ref{sec:rectrafo} by $n!^{k-1}$ should allow to analyze this class using generalized exponential functions of the type $\sum_{n \geq 0} a_n \frac{z^n}{(n!)^{k-1}}$.
This class is also interesting, since the binary case has rich combinatorial properties, such as bijection and closed-form enumeration formulas~\cite{Wallner2017Bijection}, which are also worth to investigate in the $k$-ary case.

\bibliographystyle{mybiburl}
\bibliography{bibliography}

\begin{thebibliography}{10}

\bibitem{BanderierWallner2021Walls}
Cyril Banderier and Michael Wallner.
\newblock
  \href{https://www.mat.univie.ac.at/~slc/wpapers/FPSAC2021/47.html}{{Y}oung
  {T}ableaux with {P}eriodic {W}alls: {C}ounting with the {D}ensity {M}ethod}.
\newblock  \textit{S\'{e}m. Lothar. Combin.}, 85B, Article 85B.47:12 pp., 2021.

\bibitem{Bollobas2001Graphs}
B\'{e}la Bollob\'{a}s.
\newblock \href{https://doi.org/10.1017/CBO9780511814068}{ \textit{Random
  graphs}}, volume~73 of  \textit{Cambridge Studies in Advanced Mathematics}.
\newblock Cambridge University Press, Cambridge, second edition, 2001.
\newblock
  \href{http://dx.doi.org/10.1017/CBO9780511814068}{{\textcolor{lightgray}[}doi{\textcolor{lightgray}]}}.

\bibitem{bousquet2015xml}
Mireille Bousquet-M\'{e}lou, Markus Lohrey, Sebastian Maneth, and Eric Noeth.
\newblock \href{https://doi.org/10.1007/s00224-014-9544-x}{X{ML} compression
  via directed acyclic graphs}.
\newblock  \textit{Theory Comput. Syst.}, 57(4):1322--1371, 2015.
\newblock
  \href{http://dx.doi.org/10.1007/s00224-014-9544-x}{{\textcolor{lightgray}[}doi{\textcolor{lightgray}]}}.

\bibitem{Changetal2023dcombining}
Yu-Sheng Chang, Michael Fuchs, Hexuan Liu, Michael Wallner, and Guan-Ru Yu.
\newblock \href{https://arxiv.org/abs/2209.03850}{{Enumerative and
  Distributional Results for d-combining Tree-Child Networks}}.
\newblock  \textit{arXiv preprint}, pages 1--48, 2023.

\bibitem{DowneySethiTarjan1980variations}
Peter~J. Downey, Ravi Sethi, and Robert~Endre Tarjan.
\newblock \href{http://dx.doi.org/10.1145/322217.322228}{Variations on the
  common subexpression problem}.
\newblock  \textit{J. Assoc. Comput. Mach.}, 27(4):758--771, 1980.
\newblock
  \href{http://dx.doi.org/10.1145/322217.322228}{{\textcolor{lightgray}[}doi{\textcolor{lightgray}]}}.

\bibitem{ElveyPriceFangWallner2020DFA}
Andrew Elvey~Price, Wenjie Fang, and Michael Wallner.
\newblock
  \href{https://drops.dagstuhl.de/opus/volltexte/2020/12041}{{Asymptotics of
  Minimal Deterministic Finite Automata Recognizing a Finite Binary Language}}.
\newblock In  \textit{AofA 2020}, volume 159 of  \textit{LIPIcs. Leibniz Int.
  Proc. Inform.}, pages 11:1--11:13, 2020.
\newblock
  \href{http://dx.doi.org/10.4230/LIPIcs.AofA.2020.11}{{\textcolor{lightgray}[}doi{\textcolor{lightgray}]}}.

\bibitem{ElveyPriceEtal2021Compacted}
Andrew Elvey~Price, Wenjie Fang, and Michael Wallner.
\newblock \href{https://arxiv.org/pdf/1908.11181.pdf}{Compacted binary trees
  admit a stretched exponential}.
\newblock  \textit{J. Combin. Theory Ser. A}, 177:105306, 40, 2021.
\newblock
  \href{http://dx.doi.org/10.1016/j.jcta.2020.105306}{{\textcolor{lightgray}[}doi{\textcolor{lightgray}]}}.

\bibitem{flss90}
Philippe Flajolet, Paolo Sipala, and Jean-Marc Steyaert.
\newblock \href{http://dx.doi.org/10.1007/BFb0032034}{Analytic variations on
  the common subexpression problem}.
\newblock In  \textit{Automata, languages and programming}, pages 220--234.
  Springer, 1990.
\newblock
  \href{http://dx.doi.org/10.1007/BFb0032034}{{\textcolor{lightgray}[}doi{\textcolor{lightgray}]}}.

\bibitem{GenitriniGittenbergerKauersWallner2016}
Antoine Genitrini, Bernhard Gittenberger, Manuel Kauers, and Michael Wallner.
\newblock \href{http://arxiv.org/abs/1703.10031}{Asymptotic enumeration of
  compacted binary trees of bounded right height}.
\newblock  \textit{J. Combin. Theory Ser. A}, 172:105177, 2020.
\newblock
  \href{http://dx.doi.org/10.1016/j.jcta.2019.105177}{{\textcolor{lightgray}[}doi{\textcolor{lightgray}]}}.

\bibitem{GuttmannWhittington1978Graphs}
A.~J. Guttmann and S.~G. Whittington.
\newblock \href{http://stacks.iop.org/0305-4470/11/721}{Two-dimensional lattice
  embeddings of connected graphs of cyclomatic index two}.
\newblock  \textit{J. Phys. A}, 11(4):721--729, 1978.

\bibitem{guttmann2015analysis}
Anthony~J. Guttmann.
\newblock \href{https://arxiv.org/abs/1405.5327}{Analysis of series expansions
  for non-algebraic singularities}.
\newblock  \textit{Journal of Physics A: Mathematical and Theoretical},
  48(4):045209, 2015.
\newblock
  \href{http://dx.doi.org/10.1088/1751-8113/48/4/045209}{{\textcolor{lightgray}[}doi{\textcolor{lightgray}]}}.

\bibitem{JansonEtal2000Graphs}
Svante Janson, Tomasz \L~uczak, and Andrzej Rucinski.
\newblock \href{https://doi.org/10.1002/9781118032718}{ \textit{Random
  graphs}}.
\newblock Wiley-Interscience Series in Discrete Mathematics and Optimization.
  Wiley-Interscience, New York, 2000.
\newblock
  \href{http://dx.doi.org/10.1002/9781118032718}{{\textcolor{lightgray}[}doi{\textcolor{lightgray}]}}.

\bibitem{Liskovets2006exact}
Valery~A. Liskovets.
\newblock
  \href{https://linkinghub.elsevier.com/retrieve/pii/S0166218X05002568}{Exact
  enumeration of acyclic deterministic automata}.
\newblock  \textit{Discrete Applied Mathematics}, 154(3):537--551, 2006.
\newblock
  \href{http://dx.doi.org/10.1016/j.dam.2005.06.009}{{\textcolor{lightgray}[}doi{\textcolor{lightgray}]}}.

\bibitem{McDiarmidSempleWelsh2015Networks}
Colin McDiarmid, Charles Semple, and Dominic Welsh.
\newblock \href{https://doi.org/10.1007/s00026-015-0260-2}{Counting
  phylogenetic networks}.
\newblock  \textit{Ann. Comb.}, 19(1):205--224, 2015.
\newblock
  \href{http://dx.doi.org/10.1007/s00026-015-0260-2}{{\textcolor{lightgray}[}doi{\textcolor{lightgray}]}}.

\bibitem{vdHofstad2017Graphs}
Remco van~der Hofstad.
\newblock \href{https://doi.org/10.1017/9781316779422}{ \textit{Random graphs
  and complex networks. {V}ol. 1}}.
\newblock Cambridge Series in Statistical and Probabilistic Mathematics, [43].
  Cambridge University Press, Cambridge, 2017.
\newblock
  \href{http://dx.doi.org/10.1017/9781316779422}{{\textcolor{lightgray}[}doi{\textcolor{lightgray}]}}.

\bibitem{Wallner2017Bijection}
Michael Wallner.
\newblock
  \href{http://www.sciencedirect.com/science/article/pii/S0304397518304699}{A
  bijection of plane increasing trees with relaxed binary trees of right height
  at most one}.
\newblock  \textit{Theoretical Computer Science}, 755:1--12, 2019.
\newblock
  \href{http://dx.doi.org/https://doi.org/10.1016/j.tcs.2018.06.053}{{\textcolor{lightgray}[}doi{\textcolor{lightgray}]}}.

\bibitem{Wallner2023Inversion}
Michael Wallner.
\newblock
  \href{https://2023.permutationpatterns.com/booklet/bookletco.pdf#page=142}{{D}yck
  paths and inversion tables}.
\newblock  \textit{Permutation Patterns 2023}, pages 142--144, 2023.

\end{thebibliography}

\end{document}